\documentclass[A4paper]{article}%
\usepackage{amsmath,amssymb,amsfonts}
\usepackage{theorem}
\usepackage{color}
\usepackage{hyperref}
\usepackage[font={small, it}]{caption}
\usepackage[all]{xy}
\usepackage{amsmath}
\usepackage{amsfonts}
\usepackage{amssymb}
\providecommand{\U}[1]{\protect \rule{.1in}{.1in}}
\theoremstyle{change}
\newtheorem{definition}{Definition:}[section]
\newtheorem{proposition}[definition]{Proposition:}
\newtheorem{theorem}[definition]{Theorem:}
\newtheorem{lemma}[definition]{Lemma:}

{\theorembodyfont{\rmfamily}
	
\newtheorem{remark}[definition]{Remark:}
}
{\theorembodyfont{\rmfamily}
	
\newtheorem{example}[definition]{Example:}
}
\newenvironment{proof}
{{\bf Proof:}}
{\qquad \hspace*{\fill} $\Box$}

\newcommand{\fg}{\mathfrak{g}}

\newcommand{\id}{\operatorname{id}}
\newcommand{\inner}{\operatorname{int}}

\newcommand{\rme}{\mathrm{e}}

\newcommand{\CC}{\mathcal{C}}

\newcommand{\UC}{\mathcal{U}}

\newcommand{\AC}{\mathcal{A}}
\newcommand{\XC}{\mathcal{X}}
\newcommand{\DC}{\mathcal{D}}

\newcommand{\T}{\mathbb{T}}

\newcommand{\R}{\mathbb{R}}
\newcommand{\Z}{\mathbb{Z}}
\begin{document}

\title{On the structural properties of the bounded control set of a linear control system}
\author{V\'{\i}ctor Ayala \thanks{ Supported by Proyecto Fondecyt $n^{o}$ 1190142,
Conicyt, Chile}\\Universidad de Tarapac\'a\\Instituto de Alta Investigaci\'on\\Casilla 7D, Arica, Chile\\and\\Adriano Da Silva \thanks{ Supported by Fapesp grants $n^{o}$ 2020/12971-4 and 2018/13481-0, and partially by CNPq grant no. 309820/2019-7.}\\Instituto de Matem\'atica,\\Universidade Estadual de Campinas\\Cx. Postal 6065, 13.081-970 Campinas-SP, Brasil.}
\date{\today }
\maketitle

\begin{abstract}
The present paper shows that the closure of the bounded control set of a linear control system contains all the bounded orbits of the system. As a consequence, we prove that the closure of this control set is the continuous image of the cartesian product of the set of control functions by the central subgroup associated with the drift of the system.
\end{abstract}



\section{Introduction}

Let $\XC, Y^1, \ldots, Y^m$ be smooth vector fields on a connected finite dimensional differentiable manifold $M$. A control system $\Sigma_M$ on $M$
is determined by a family of controlled differential equations \
\begin{flalign*}
&&\dot{x}(t)=\XC(x(t))+\sum_{j=1}^mu_j(t)Y^j(x(t)),  &&\hspace{-1cm}\left(\Sigma_M\right)
\end{flalign*}
which allows changing the behavior of $\XC$ according to the control vectors
$Y^{1},Y^{j},...,Y^{m}$, and the set $\mathcal{U}$ of admissible control
functions
\[
\mathcal{U}=\left \{  u:\mathbb{R}\rightarrow \mathbb{R}^{m}%
\ :\ u\mbox{ is measurable with }u(t)\in \Omega \mbox{ a.e.}\right \}  ,
\]
where $\Omega \subset \mathbb{R}^{m}$ is a closed and convex set with
$0\in \operatorname{int}\Omega$.

The controllability notion of $\Sigma_M$ is one of the most relevant properties
of the system. It allows connecting any two points of $M$ through a
concatenation of integral curves of $\Sigma_M$, in positive time. For instance,
a necessary condition to solve any optimization problem between two states,
like a time optimal or minimum energy trajectory, is the existence of a
control $u\in \mathcal{U}$ such that the corresponding solution connects these
two states. Despite the existence of nice examples of controllable systems, 
this global property is hard to be satisfied in general. A well known example is
the class of linear control system on Euclidean spaces $\Sigma_{\mathbb{R}^{n}}$, where $\XC=A$ is a matrix of order $n$, and $Y^{j}=b^{j}\in\R^n$ are constant vector fields. In this context, the Kalman rank
condition characterizes controllability. However, to obtain that, you need to
consider $\Omega=\mathbb{R}^{m}$, which is far from real life. A more
realistic approach considers the case when $\Omega \subset \mathbb{R}^{m}$ is a
compact subset, and the notion of control sets for $\Sigma_M$, which are roughly speaking maximal subsets of $M$
where controllability holds. Recently, several papers are focused on studing the controllability
and the existence, uniqueness and topological properties of the control sets
of linear control systems on connected Lie groups. For this case, the drift $\mathcal{X}$ is a linear vector
field in the sense that its flow $\{ \varphi_{t}\}_{t\in \mathbb{R}}$ is a
$1$-parameter group of automorphisms of $G$, and the control vectors are elements
of the Lie algebra. It turns out that many properties of the system depend strongly of the
dynamical behavior of $\mathcal{X}.$ In fact, the flow of $\XC$ induce connected subgroups, called unstable, central and stable, which have a nice relationship with the set of the reachable points of the system and hence with its controllability and control sets (see \cite{DSAy0, DSAy, DSAyGZ, DS, JPh1, JPh2}). 

Our aim here is to study the structural properties of the control set of a linear control system that contains the identity of the group. If such control set is bounded and contains the identity element in its interior, our main result shows that all bounded orbits of $\Sigma_G$ are contained in its closure. As a consequence, such closure is the continuous image of the cartesian product of $\UC$ by the central subgroup.

The article is organized as follows: In Section 2 we introduce the decompositions induced by automorphisms and the notion of linear control systems and state the main results relating both notions. We finish the section with some new results needed in the prove of our main results. Section 3 is used to introduce and prove our main result. We prove here several properties concerning bounded orbits and the central subgroup of $\XC$ which implies the main result. Moreover, such properties also allows us to create a continuous function from the cartesian of $\UC$ by the central subgroup in the closure of the control set, showing that under our hypothesis the control set is the continuous image of this cartesian.

\subsection*{Notations}
In the present paper, all the Lie groups and algebras considered are real and finite dimensional. Let $G$ be a connected Lie group. A subgroup $H\subset G$ i said to be trivial if $H=\{e\}$, where $e\in G$ stands for the identity element of $G$. By $L_g$ and $R_g$ we denote, respectively, the left and right-translations by $g$. The conjugation of $g$ is the map $C_g:=L_g\circ R_{g^{-1}}$. The center $Z(G)$ of $G$ is the set of elements in $G$ that satisfy $C_g=\id_G$. If $f:G\rightarrow H$ is a differentiable map between Lie groups, the differential of $f$ at $x$ is denoted by $(df)_x$.

\section{Preliminaries}

This section is devoted to present the background needed in order to establish our main results. We introduce here the decompositions induced on Lie groups and algebras by their automorphisms and also the notion of LCS and their control sets. The main results needed are also stated here. In the end of the section we
prove some new results that will be necessary at Section 3. 

\subsection{Decomposition by automorphisms}

Let $\mathfrak{g}$ be a finite dimensional Lie algebra and $\rho \in \mathrm{Aut}(\mathfrak{g})$ an automorphism. For any eigenvalue $\alpha \in \mathbb{C}$ of $\rho$, the real generalized
eigenspaces of $\rho$ associated with $\alpha$ are given by
\[
\mathfrak{g}_{\alpha}=\{X\in \mathfrak{g}:(\rho-\alpha I)^{n}%
X=0\; \; \mbox{for some }n\geq1\}, \; \; \mbox{ if }\; \; \alpha \in \mathbb{R}
\; \; \mbox{ and}
\]
\[
\mathfrak{g}_{\alpha}=\mathrm{span}\{ \mathrm{Re}(v), \mathrm{Im}%
(v);\; \;v\in \bar{\mathfrak{g}}_{\alpha}\},\; \; \mbox{ if }\; \; \alpha
\in \mathbb{C}%
\]
where $\bar{\mathfrak{g}}=\mathfrak{g}+i\mathfrak{g}$ is the complexification
of $\mathfrak{g}$ and $\bar{\mathfrak{g}}_{\alpha}$ the generalized eigenspace
of $\bar{\rho}=\rho+i\rho$, the extension of $\rho$ to $\bar{\mathfrak{g}}$.

We define the {\bf unstable, central }and {\bf stable} $\rho$-invariant
subspaces of $\mathfrak{g}$, respectively, by
\[
\mathfrak{g}^{+}=\bigoplus_{\alpha:\, |\alpha|>1}\mathfrak{g}_{\alpha}%
,\hspace{1cm}\mathfrak{g}^{0}=\bigoplus_{\alpha:\, |\alpha|=1}\mathfrak{g}%
_{\alpha}\hspace{1cm}\mbox{ and }\hspace{1cm}\mathfrak{g}^{-}=\bigoplus
_{\alpha:\, |\alpha|<1}\mathfrak{g}_{\alpha}.
\]

Following \cite{DSAyHR} the fact that $[\bar{\mathfrak{g}}_{\alpha}%
,\bar{\mathfrak{g}}_{\beta}]\subset \bar{\mathfrak{g}}_{\alpha \beta}$ when
$\alpha \beta$ is an eigenvalue of $\rho$ and zero otherwise implies that
$\mathfrak{g}^{+},\mathfrak{g}^{0},\mathfrak{g}^{-}$ are in fact $\rho
$-invariant Lie subalgebras with $\mathfrak{g}^{+}$, $\mathfrak{g}^{-}$
nilpotent ones. Moreover, $\fg$ is decomposed as $\fg=\mathfrak{g}^{+}\oplus\mathfrak{g}^{0}\oplus\mathfrak{g}^{-}$.

At the group level, let $G$ be a connected Lie group with Lie algebra $\fg$. For any automorphism $\psi \in \mathrm{Aut}(G)$ the {\bf dynamical subgroups} of $G$ induced by $\psi$ are the Lie subgroups defined as follows: Since $(d\psi)_e$ is an automorphism of $\fg$ it induces subalgebras $\fg^+$, $\fg^-$ and $\fg^0$ of $\fg$ defined as previously. We can consider then connected subgroup $G^+, G^-$ and $G^0$ with the Lie algebras $\fg^+$, $\fg^-$ and $\fg^0$, respectively. As previously, the subgroups $G^+$, $G^0$ and $G^-$ are called the {\bf unstable, central} and {\bf stable} subgroups of $G$, respectively. 

Also, we denote by $G^{+, 0}$ and $G^{-, 0}$ the connected subgroups of $G$ with Lie algebras given by $\mathfrak{g}^{+, 0}:=\mathfrak{g}^{+}\oplus \mathfrak{g}^{0}$ and
$\mathfrak{g}^{-, 0}:=\mathfrak{g}^{-}\oplus \mathfrak{g}^{0}$, respectively, and by $G^{+, -}$ the product $G^{+, -}=G^+G^-$. 

We say that $G$ is {\bf decomposable} by $\varphi$ if
\[
G=G^{+, 0}G^{-}=G^{-, 0}G^{+}=G^{+, -}G^{0}.
\]
If $G^{0}=\{e\}$ we say that $\psi$ is {\bf hyperbolic}. In particular, if
$\psi$ is hyperbolic, $G=G^{+, -}:=G^{+}G^{-}$ is decomposable.

The next proposition summarizes the main properties of the previous subgroups. Its proof can be found in \cite[Proposition 3.4]{DSAyHR}.

\begin{proposition}
	\label{prop}
	For the dynamical subgroups of an automorphism $\psi$ of $G$, it holds:
	\begin{enumerate}
		\item If $G$ is solvable or if $G^0$ is a compact subgroup, then $G$ is decomposable; 
		\item If $\psi$ is hyperbolic, then $G$ is a nilpotent Lie group.
	\end{enumerate}

\end{proposition}

The previous decompositions for single automorphisms of $G$ can be extended to flows of automorphisms as follows: Let $\{ \varphi_{t}\}_{t\in \mathbb{R}}$ be a $1$-parameter flow of automorphisms on $G$. By derivation, $\{(d\varphi_t)_e, t\in\R\}$ is a $1$-parameter subgroup of $\mathrm{Aut}(\fg)$ and hence, there exists a derivation $\DC$ of $\mathfrak{g}$ such that $(d\varphi_t)_e=\rme^{t\DC}$ for any $t\in\R$. In particular, the subalgebras induced by $(d\varphi_t)_e$ coincide, for any $t\in\mathbb{R}$ with the sum of the generalized real eigenspaces of $\DC$ associated to eigenvalues with positive, zero and negative real parts. Therefore, we can define the dinamical subgroups of $\{ \varphi_{t}\}_{t\in \mathbb{R}}$ to be the dinamical subgroups of  $\varphi_{\tau}$ for some (and hence any) $\tau\in\R$. The properties of these subgroups where studied in previous works (\cite{DSAy, DS, DSAyGZ}) and their topological properties are much nicer than the ones associated with single automorphisms as the next result shows.

\begin{proposition}
	\label{propcont}
	Let $G^+, G^0$ and $G^-$ be the dynamical subgroups of the $1$-parameter flow of automorphisms $\{\varphi_{t}\}_{t\in \mathbb{R}}$. It holds:
	\begin{enumerate}
	\item $G^+, G^0$ and $G^-$ are closed and have trivial intersection;
	\item $G^+$ and $G^-$ are simply connected nilpotent Lie subgroups;
	\item Let $N$ be the nilradical of $G$ and assume that $G^0$ is a compact subgroup. Then
	\subitem 2.1. $G^{+, -}\subset N$;
	\subitem 2.2. $N^0=N\cap G^0$ is a compact, connected normal subgroup of $G$.
	
\end{enumerate}
	
\end{proposition}

The next example shows that the dynamical subgroups associated with single automorphisms does not need to be closed.

\begin{example}
	\label{toro}
	Let us consider $\mathbb{T}^2=\R^2/\Z^2$ the $2$-dimensional torus. As a general result, the group of automorphisms $\mathrm{Aut}(\T^2)$ is given by 
	$$\mathrm{Aut}(\T^2)=\{A\in\mathrm{Gl}(\R^2); \;\; A\Z^2=\Z^2\}.$$
	In particular, 
	$$\psi=\left(\begin{array}{cc}
	1 & 1 \\ 2 & 1
	\end{array}\right)\in\mathrm{Aut}(\T^2), \;\;\;\mbox{ with }\;\;\;\fg^+=\R(1, \sqrt{2})\;\;\;\mbox{ and }\;\;\;\fg^-=\R(1, -\sqrt{2}).$$
	Moreover, both $G^+$ and $G^-$ are images of the well known irrational flows on $\T^2$, and are henceforth not closed in $\T^2$. 
\end{example}

\subsection{Linear control systems}

Let $G$ be a connected Lie group with Lie algebra $\mathfrak{g}$ identified
with the set right-invariant vector fields of $G$. A {\bf linear control system (LCS)} on $G$ is given by
a family of ordinary differential equations
\begin{flalign*}
&&\dot{g}(t)=\XC(g(t))+\sum_{j=1}^mu_j(t)Y^j(g(t)),  &&\hspace{-1cm}\left(\Sigma_G\right)
\end{flalign*}
where the {\bf drift} $\mathcal{X}$ is a linear vector field, that is, its
associated flow $\{ \varphi_{t}\}_{t\in \mathbb{R}}$ is a $1$-parameter group of
automorphisms, $Y^{1}, \ldots, Y^{m}\in \mathfrak{g}$ and $u=(u_{1}, \ldots,
u_{m})\in \mathcal{U}$. The set $\mathcal{U}$ of admissible {\bf control functions}
is given by%
\[
\mathcal{U} = \left \{ u:\mathbb{R}\rightarrow \mathbb{R}^{m}\ :\ u
\mbox{ is measurable with } u(t) \in \Omega \mbox{ a.e.}\right \} ,
\]
where $\Omega$ is a compact and convex subset of $\mathbb{R}^{m}$ with $0
\in \operatorname{int} \Omega$. Endowed with the
weak$^{*}$-topology of $L^{\infty}(\mathbb{R},\mathbb{R}^{m}) = L^{1}%
(\mathbb{R},\mathbb{R}^{m})^{*}$ the set $\mathcal{U}$ is a compact metrizable space and the {\bf shift
flow}
\[
\theta:\mathbb{R} \times \mathcal{U} \rightarrow \mathcal{U},\quad(t,u)
\mapsto \theta_{t}u = u(\cdot+ t),
\]
is a continuous dynamical system (see \cite[Section
4.2]{FCWK}). By $\mathcal{U}_{\mathrm{per}}$ we denote the set of
periodic points of $\theta$ in $\mathcal{U}$, that is, $u\in \mathcal{U}%
_{\mathrm{per}}$ if there exists $\tau>0$ such that $\theta_{\tau}u=u$. By
\cite[Lemma 4.2.2]{FCWK} it holds that $\mathcal{U}_{\mathrm{per}}$ is dense in
$\mathcal{U}$.

For any $g\in G$ and $u\in \mathcal{U}$, the solution $t\mapsto \phi(t, g, u)$
of $\Sigma_{G}$ is defined for the whole real line and satisfies the {\bf cocycle property}
\[
\phi(t+s, g, u)=\phi(t,\phi(s, g, u),\theta_{s}u)
\]
for all $t,s\in \mathbb{R}$, $g\in G$, $u\in \mathcal{U}.$ The {\bf control flow} of the system is the skew-product flow%
\[
\Phi:\mathbb{R} \times \mathcal{U} \times G \rightarrow \mathcal{U} \times
G,\quad(t,u,x) \mapsto \Phi_{t}(u,x) = (\theta_{t}u,\phi(t,x,u)).
\]
We also write $\phi_{t,u}:G\rightarrow G$ for the map $g\mapsto \phi(t, g, u)$.
In the particular case of LCS, the intrinsic relations between the
vector fields involved and the group structure implies the following relation
\begin{equation}
\label{solucaolinear}\phi_{\tau, u}\circ R_{g}=R_{\varphi_{\tau}(g)}\circ
\phi_{\tau, u}, \; \; \mbox{ for any }\; \; \tau \in \mathbb{R}, g\in G,
\end{equation}
where $\{ \varphi_{t}\}_{t\in \mathbb{R}}$ is the flow of $\mathcal{X}$. We
define the set of points reachable from $g\in G$ at time $\tau>0$ and
the {\bf reachable set of $g$}, respectively, by%
\[
\mathcal{A}_{\tau}(g):=\{ \phi_{\tau, u}(g), \; \; \; u\in \mathcal{U}\}
\; \; \; \mbox{ and }\; \; \; \mathcal{A}(g):=\bigcup_{\tau>0}\mathcal{A}_{\tau
}(g).
\]
We denote by $\mathcal{A}_{\tau}^{*}(g)$ and $\mathcal{A}^{*}(g)$ the respective reachable sets in negative time.

Next we define control sets for a LCS.

\begin{definition}
	\label{controlset}  A nonempty subset $\mathcal{C}\subset G$ is said to be a
	{\bf control set } of $\Sigma_{G}$ if it is maximal (w.r.t. set inclusion)
	satisfying 
	
	\begin{itemize}	
		\item[(i)] For each $g\in \mathcal{C}$ there exists $u\in \mathcal{U}$ such that
		$\phi(\mathbb{R}^{+}, g, u)\subset \mathcal{C}$; 
		
		\item[(ii)] For any $g\in \mathcal{C}$ it holds that $\mathcal{C}%
		\subset \overline{\mathcal{A}(g)}$. 
	\end{itemize}
\end{definition}

Since the identity $e\in G$ is a singularity of $\mathcal{X}$ and $0\in \operatorname{int}\Omega$, there exists a control set $\mathcal{C}$ of
$\Sigma_{G}$ containing the identity. Moreover, $e\in \operatorname{int}\mathcal{C}$ if and only if $\AC(e)$ (or equivalently $\AC^*(e)$) is open. In this case, it holds that 
$$\CC=\overline{\AC(e)}\cap\AC^*(e).$$

The next result relates the dinamics of a LCS with the dinamical subgroups associated with the flow of $\mathcal{X}$ (see \cite[Lemma 3.1]{DS} and \cite[Theorem 3.8]{DSAy}).

\begin{theorem}
	\label{properties}
	For a LCS $\Sigma_G$, it holds:
\begin{itemize}

\item[1.] If $\varphi_{t}(g)\in\mathcal{A}(e)$ for all $t\in \mathbb{R}$ then $\mathcal{A}(e)g\subset
\mathcal{A}(e)$; 

\item[2.] If $\varphi_{t}(g)\in\mathcal{A}^*(e)$ for all $t\in \mathbb{R}$ then $\mathcal{A}^*(e)g\subset
\mathcal{A}^*(e)$; 

\item[3.] If $\mathcal{A}(e)$ is open, then 
\subitem 3.1. $G^{+, 0}\subset \mathcal{A}(e)$ and $G^{-, 0}\subset \mathcal{A}^{*}(e)$;
\subitem 3.2. If $G$ is decomposable, $\CC$ is the only control set with nonempty interior of $\Sigma_G$;
\subitem 3.3. $G^0$ is a compact subgroup if and only if $\CC$ is bounded.
\end{itemize}
\end{theorem}

\subsection{Some results}

In this section, we obtain some results related to the dynamical subgroups of automorphisms that will be necessary in the proof of the main result.

\begin{proposition}
	\label{difeo}  If $\psi \in \mathrm{Aut}(G)$ is hyperbolic, the map
	\[
	f_{\psi}: \; \; \;g\in G\mapsto g\psi(g^{-1})\in G
	\]
	is an onto local diffeomorphism. Furthermore, $f_{\psi}$ is injective if and only
	if the only fixed point of $\psi$ is the identity element of $G$.
\end{proposition}

\begin{proof}
	The differentiability of $f_{\psi}$ follows directly from its definition.
	Moreover, for any $G$ it holds that
	\begin{equation}
	\label{1}f_{\psi}(hg)=hg\psi((hg)^{-1})=h\left( g\psi(g^{-1})\right)
	\psi(h^{-1})=hf_{\psi}(g)\psi(h^{-1})\implies f_{\psi}\circ L_{h}=L_{h}\circ
	R_{\psi(h^{-1})}\circ f_{\psi},
	\end{equation}
	and since left and right translations are diffeomorphims we get in particular
	that $f_{\psi}$ has constant rank. On the other hand,
	\[
	(df_{\psi})_{e}X=\frac{d}{dt}_{|t=0}\mathrm{e}^{tX}\psi \left( \mathrm{e}%
	^{-tX}\right) =\left( (dL_{\mathrm{e}^{tX}})_{\psi(\mathrm{e}^{-tX})}\frac
	{d}{dt}\psi(\mathrm{e}^{-tX})+(dR_{\psi(\mathrm{e}^{-tX})})_{\mathrm{e}^{tX}%
	}\frac{d}{dt}\mathrm{e}^{tX}\right) _{|t=0}=X-(d\psi)_{e}X
	\]
	and since $\psi$ is hyperbolic, we have that $(d\psi)_{e}$ does not have $1$ as an eigenvalue. Therefore, $(df_{\psi
	})_{e}(\mathfrak{g})=\mathfrak{g}$ and hence $f_{\psi}$ is a local diffeomorphism.
	
	The surjectiveness of $f_{\psi}$ is proved by induction on the dimension of $G$. If $G$ is an abelian group, for any $g, h\in G$, we get from equation
	(\ref{1}) that
	\[
	f_{\psi}(gh)=f_{\psi}(g)f_{\psi}(h)
	\]
	showing that $f_{\psi}$ is an homomorphism. Being $f_{\psi}$ a local
	diffeomorphism we must have that $f_{\psi}(G)$ is a subgroup with nonempty
	interior in $G$ and hence $f_{\psi}(G)=G$. In particular, the result is true
	when $\dim G=1$.
	
	Assume now that the result is true for any Lie group $G$ with $\dim G<n$ and
	consider $G$ to be a nonabelian Lie group of dimension $n$ and $\psi
	\in \mathrm{Aut}(G)$ to be hyperbolic. Since $G$ is a nilpotent Lie group, the connected component $Z(G)_0$ of its center is a closed, connected, nontrivial, normal subgroup of $G$ and so $\hat{G}=G/Z(G)_0$ is a connected Lie group satisfying
	$\dim \hat{G}=\dim G-\dim Z(G)_0<n$. Let then $\hat{\psi}\in \mathrm{Aut}(\hat
	{G})$ such that
	\[
	\hat{\psi}\circ \pi=\pi \circ \psi, \; \; \mbox{ where }\; \; \pi:G\rightarrow \hat
	{G}\; \mbox{ is the canonical projection}.
	\]
	A simple calculation shows that the associated map
	\[
	f_{\hat{\psi}}:\hat{g}\in \hat{G}\mapsto \hat{g}\hat{\psi}(\hat{g}^{-1})\in
	\hat{G}%
	\]
	also verifies $f_{\hat{\psi}}\circ \pi=\pi \circ f_{\psi}$. Since $\pi(G^{+, -})=\hat{G}^{+, -}$ and $\pi(G^{0})=\hat{G}^{0}$ it holds that $\hat{\psi}$ is hyperbolic and by the
	inductive hypothesis
	\begin{equation}
	\label{2}\pi(f_{\psi}(G))= f_{\hat{\psi}}(\pi(G))=f_{\hat{\psi}}(\hat{G}%
	)=\hat{G}=\pi(G).
	\end{equation}
	On the other hand, since $\dim Z(G)_{0}<n$ and $\psi|_{Z(G)_{0}}%
	\in \mathrm{Aut}(Z(G)_{0})$ is hyperbolic,  by the abelian case we get
	\[
	f_{\psi|_{Z(G)_{0}}}=f_{\psi}|_{Z(G)_{0}}\; \; \mbox{ is surjective}.
	\]
	Therefore, for any $g\in G$ there are by (\ref{2}) elements $x\in G$ and
	$h\in Z(G)_{0}$ such that $g=f_{\psi}(x)h.$ Since $f_{\psi}|_{Z(G)_{0}}$ is
	surjective it holds that $h=f_{\psi}(y)$ for some $y\in Z(G)_{0}$ and so
	\[
	g=f_{\psi}(x)h=f_{\psi}(x)f_{\psi}(y)=f_{\psi}(xy),
	\]
	where for the last equality we used (\ref{1}) and the fact that $y\in
	Z(G)_{0}$. Therefore, $f_{\psi}$ is surjective as stated.
	
	For the last assertion, note that
	\[
	f_{\psi}(g)=f_{\psi}(h)\iff g\psi(g^{-1})=h\psi(h^{-1})\iff \psi(g^{-1}%
	h)=g^{-1}h
	\]
	and consequently $f_{\psi}$ is injective if and only $e\in G$ is only fixed point of $\psi$ as stated.
\end{proof}

\bigskip

The next result shows that for the dynamical subgroups associated with a flow of automorphisms, there is a diffeomorphism between the sum of the stable and unstable subalgebras and their respective subgroups that conjugates the flow and its differential.

\begin{proposition}
	\label{map}  For any $1$-parameter flow of autormorphisms $\{ \varphi_{t}\}_{t\in \mathbb{R}}$, the map
	$$f:\fg^+\oplus\fg^-\rightarrow G^{+, -}, \;\;\;\; X+Y\mapsto \mathrm{e}^{X}\mathrm{e}^{Y},$$
	is a diffeomorphism and 
	\begin{equation}
	\forall t\in \mathbb{R}, \; \; \; f\circ(d\varphi_{t})_{e}=\varphi_{t}\circ f.
	\end{equation}
	
\end{proposition}

\begin{proof}
	The differentiability of $f$ comes direct from its definition. Moreover, by Proposition \ref{propcont}, the Lie subgroups $G^+$ and $G^-$ associated with $\{\varphi_t\}_{t\in\R}$ are simply connected and hence the exponential map restricted to both $\mathfrak{g}^{+}$ and $\mathfrak{g}^{-}$ are
	diffemorphisms. Let $X\in \mathfrak{g}^{+}$ and define the
	isomorphism
	\[
	T^{+}_{X}:\mathfrak{g}^{+}\rightarrow \mathfrak{g}^{+}, \; \; \;T^{+}%
	_{X}:=(dL_{\mathrm{e}^{-X}})_{\mathrm{e}^{X}}\circ(d\exp|_{\mathfrak{g}^{+}%
	})_{X}.
	\]
	For $Z\in \mathfrak{g}^{+}$, the curve
	\[
	\gamma:\R\rightarrow \mathfrak{g}^{+}, \; \; \; \; \gamma(s):=\exp
	|_{\fg^+}^{-1}\left( \mathrm{e}^{X}\mathrm{e}^{sT^{+}_{X}Z}\right) ,
	\]
	is well defined, differentiable and satisfies $\gamma(0)=X$. Moreover, 
	\[
	(d\exp|_{\mathfrak{g}^{+}})_{X}\gamma^{\prime}(0)=\frac{d}{ds}_{|s=0}%
	\mathrm{e}^{\gamma(s)}=\frac{d}{ds}_{|s=0}\mathrm{e}^{X}\mathrm{e}^{sT_{X}%
		Z}=(dL_{\mathrm{e}^{X}})_{e}T^{+}_{X}Z,
	\]
	and hence $\gamma^{\prime}(0)=Z$. Analogously, for $Y, W\in \mathfrak{g}^{-}$ we construct a curve $\beta:\R\rightarrow \mathfrak{g}^{-}$ satisfying
	\[
	\mathrm{e}^{\beta(s)}=\mathrm{e}^{sT^{-}_{Y}W}\mathrm{e}^{Y}, \; \; \; \beta
	(0)=Y\; \; \; \mbox{ and }\; \; \; \beta^{\prime}(0)=W.
	\]
	Therefore,
	\[
	(df)_{X+Y}(Z+W)=\frac{d}{ds}_{|s=0}f(\gamma(s)+\beta(s))=\frac{d}{ds}%
	_{|s=0}\mathrm{e}^{\gamma(s)}\mathrm{e}^{\beta(s)}%
	\]
	\[
	=\frac{d}{ds}_{|s=0}\mathrm{e}^{X}\mathrm{e}^{sT^{+}_{X}Z}\mathrm{e}%
	^{sT^{-}_{Y}W}\mathrm{e}^{Y}=\left( d(L_{\mathrm{e}^{X}}\circ R_{\mathrm{e}%
		^{Y}})\right) _{e}(T^{+}_{X}Z+T^{-}_{Y}W),
	\]
	which shows that $f$ is a local diffeomorphism.
	
	For the injectivity, let $X, Z\in \mathfrak{g}^{+}$ and $Y, W\in \mathfrak{g}^{-}$. Then,
	\[
	f(X+Y)=f(Z+W)\; \; \; \iff \; \; \; \mathrm{e}^{X}\mathrm{e}^{Y}=\mathrm{e}%
	^{Z}\mathrm{e}^{W}\; \; \; \iff \; \; \;G^{+}\ni \mathrm{e}^{-Z}\mathrm{e}%
	^{X}=\mathrm{e}^{W}\mathrm{e}^{-Y}\in G^{-}.
	\]
	Since $G^{+}$ and $G^{-}$ are associated with a $1$-parameter group of
	automorphisms, it holds that $G^{+}\cap G^{-}=\{e\}$ and hence
	$\mathrm{e}^{X}=\mathrm{e}^{Z}$ and $\mathrm{e}^{Y}=\mathrm{e}^{W}$ implying that $X=Z$ and $Y=W$ and showing the injectiveness of $f$.
	
	For the last assertion, the fact that $\mathfrak{g}^{+}$ and $\mathfrak{g}^{-}$ are $(d\varphi_{t}%
	)$-invariant implies
	\[
	f\left( (d\varphi_{t})_{e}(X+Y)\right) =f\left( (d\varphi_{t})_{e}%
	X+(d\varphi_{t})Y\right) =\mathrm{e}^{(d\varphi_{t})_{e}X}\mathrm{e}%
	^{(d\varphi_{t})_{e}Y}%
	\]
	\[
	=\varphi_{t}(\mathrm{e}^{X})\varphi_{t}(\mathrm{e}^{Y})=\varphi_{t}%
	(\mathrm{e}^{X}\mathrm{e}^{Y})=\varphi_{t}(f(X+Y)),
	\]
	ending the proof. 
\end{proof}

\bigskip

The previous proposition implies the following lemma.

\begin{lemma}
	\label{explosion}  Let $\{ \varphi_{t}\}_{t\in \mathbb{R}}$ be a $1$-parameter
	group of automorphisms. For any compact subset $K\subset G$ and any $x\in
	G^{+, -}$ with $x\neq e$ there exists $t_{0}\in \mathbb{R}$ such that
	$\varphi_{t}(x) \notin K$ for $t\geq t_{0}$.
\end{lemma}

\begin{proof}
	By Proposition \ref{map} any $x\in G^{+, -}$ can be written as $x=f(X+Y)$, where $X\in\fg^+$ and $Y\in\fg^-$. Moreover,  
	$$\varphi_t(x)=\varphi_t(f(X+Y))=f((d\varphi_t)_e)(X+Y),$$
	and since $f$ is a diffeomorphism, it is enough to show the equivalent assertion for
	$(d\varphi_{t})_{e}$ restricted to $\mathfrak{g}^{+, -}$. Following \cite{DS1}, there exists $c, \mu>0$ such that 
	$$|(d\varphi_t)_ev|\geq c\rme^{\mu t}|v|, \;\;t>0, v\in\fg^+ \;\;\;\;\mbox{ and }\;\;\;\;|(d\varphi_t)_ev|\geq c\rme^{\mu t}|v|, \;\;t<0, v\in\fg^-,$$
	which implies the assertion.
\end{proof}

\section{The main result and its consequences}

Throughout the whole section, we will assume that 
\begin{flalign*}
&&\dot{g}(t)=\XC(g(t))+\sum_{j=1}^mu_j(t)Y^j(g(t)),  &&\hspace{-1cm}\left(\Sigma_G\right)
\end{flalign*}
is a LCS such that the subgroup $G^{0}$ associated with $\XC$ is a compact subgroup. Moreover, we will assume that the control set $\mathcal{C}$ containing the identity element of $G$ satisfies $e\in \inner\CC$.

Our aim here is to study the structural properties of $\mathcal{C}$ under the previous hypothesis.

\subsection{The main result}

From here on let us denote by $\varrho$ a left-invariant metric of $G$ compatible with the topology. The next result characterizes bounded orbits of LCS. 

\begin{proposition}
\label{unique}  For any $u\in \mathcal{U}$ there exists at most one $x\in G^{+,
-}$ such that
\[
\left \{ \phi_{t, u}(x), \; \;t\in \mathbb{R}\right \} \; \mbox{ is bounded}.
\]

\end{proposition}

\begin{proof}
In fact, if $x_{1}, x_{2}\in G^{+, -}$ are such that $\left \{ \phi_{t,
u}(x_{i}), \; \;t\in \mathbb{R}\right \} $ is bounded for $i=1, 2$ then
\[
\varrho \left( \phi_{t, u}(x_{1}), \phi_{t, u}(x_{2})\right) =\varrho \left(
\varphi_{t}(x_{1}), \varphi_{t}(x_{2})\right) =\varrho(\varphi_{t}(x_{2}%
^{-1}x_{1}), e)
\]
is bounded, which by Lemma \ref{explosion} implies necessarely $x_{2}^{-1}%
x_{1}=e$ or $x_{1}=x_{2}$.
\end{proof}

\bigskip

Next we show that any element of $\mathcal{U}_{\mathrm{per}}$ is associated
with a bounded orbit.

\begin{proposition}
\label{x(u)}  With the previous assumptions, for any $u\in \mathcal{U}%
_{\mathrm{per}}$ there exists a unique $x(u)\in G^{+, -}$ such
\[
\phi_{n\tau, u}(x(u))\in x(u)G^{0}, \; \; \mbox{ for all }\; \;n\in \mathbb{Z}.
\]
In particular, the orbit $\left \{ \phi_{t, u}(x(u)), \; \;t\in \mathbb{R}%
\right \} $ is bounded.
\end{proposition}

\begin{proof}
Assume first that $G^{+, -}$ is a subgroup of $G$. 

Let $u\in \mathcal{U}%
_{\mathrm{per}}$ be a $\tau$-periodic control function and decompose $\phi_{\tau,
u}(e)=xy$ with $x\in G^{+, -}$ and $y\in G^{0}$. Since $G^{0}$ normalizes
$G^{+, -}$ we have that $\psi:=C_{y}\circ \varphi_{\tau}|_{G^{+, -}}$ is an
automorphism of $G^{+, -}$. Moreover, the compactness of $G^{0}$ implies that
$G^{+}$ and $G^{-}$ are the unstable and stable subgroups of $\psi$ and hence
$\psi$ is hyperbolic. Also, a simple calculation shows that
\[
\psi^{n+1}(x)=\left( y\varphi_{\tau}(y)\cdots \varphi_{n\tau}(y)\right)
\varphi_{(n+1)\tau}(x)\left( y\varphi_{\tau}(y)\cdots \varphi_{n\tau}(y)\right)
^{-1}%
\]
implying that
\[
\psi(x)=x \; \; \; \iff \; \; \;(\varphi_{n\tau}(x))_{n\in \mathbb{N}%
}\; \; \; \mbox{ is bounded.}
\]
However, as a consequence of Lemma \ref{explosion}, the only point $x\in G^{+,
-}$ such that $(\varphi_{n\tau}(x))_{n\in \mathbb{N}}$ is bounded is the
identity $e\in G$.

By Proposition \ref{difeo}, the map
\[
f_{\psi}:\; \; \;g\in G^{+, -}\mapsto g\psi(g^{-1})\in G^{+, -}%
\; \; \; \mbox{ is a diffeomorphism}.
\]
Let then $x(u)\in G^{+, -}$ be the unique element satisfying $x=f_{\psi}(x(u))$. It holds that
\[
x=f_{\psi}(x(u))=x(u)\psi(x(u)^{-1})=x(u)C_{y}(\varphi_{\tau}(x(u)^{-1}%
))=x(u)y\varphi_{\tau}(x(u)^{-1})y^{-1}%
\]
and hence
\[
\phi_{\tau, u}(x(u))=\phi_{\tau, u}(e)\varphi_{\tau}(x(u))=xy\varphi_{\tau
}(x(u))=x(u)y\in x(u)G^{0}.
\]
Also, if $\phi_{n\tau, u}(x(u))\in x(u)G^{0}$ then
\[
\phi_{(n+1)\tau, u}(x(u))=\phi_{\tau, u}(\phi_{n\tau, u}(x(u))\in \phi_{\tau,
u}\left( x(u)G^{0}\right) =\phi_{\tau, u}(x(u))\cdot \varphi_{\tau}(G^{0})\in
x(u)G^{0},
\]
where for the last equality we used equation (\ref{solucaolinear}). Let
$n\in \mathbb{Z}^{+}$ and $g\in G^{0}$ such that $\phi_{n\tau, u}(x(u))=x(u)g$.
Since $\phi_{\tau, u}^{-1}=\phi_{-\tau, \theta_{t}u}=\phi_{-\tau, u}$ we have
that
\[
x(u)=\phi_{-n\tau, u}\left( \phi_{n\tau, u}(x(u))\right) =\phi_{-n\tau,
u}(x(u)g)=\phi_{-n\tau, u}(x(u))\varphi_{-n\tau}(g)
\]
\[
\implies \phi_{-n\tau, u}(x(u))=x(u)\varphi_{-n\tau}(g^{-1})\in x(u)G^{0}%
\]
and hence
\[
\phi_{n\tau, u}(x(u))\in x(u)G^{0}, \; \; \mbox{ for all }\; \;n\in \mathbb{Z}.
\]

For the general case, let us consider the induced LCS $\Sigma_{\hat{G}}$ on the Lie group $\hat{G}=G/N^{0}$, where $N^0$ is the subgroup defined in Proposition \ref{propcont} given by the intersection of $G^0$ with the nilradical $N$ of $G$. Since $G^0$ is compact, we have by the same proposition that $G^{+, -}\subset N$ and hence $\hat{G}^{+, -}=\pi(G^{+, -})=\pi(G^{+, -}N^{0})=\pi(N)$ implying that $\hat{G}^{+, -}$ is a subgroup. By the previous
case, for any $u\in \mathcal{U}_{\mathrm{per}}$ there exists a unique $\hat{x}(u)\in
\hat{G}^{+, -}$ such that
\[
\hat{\phi}_{n\tau, u}(\hat{x}(u))\in \hat{x}(u)\hat{G}^{0}, \; \;n\in \mathbb{Z}.
\]
Since $G$ is decomposable, there exists a unique $x(u)\in G^{+, -}$ such that $\pi(x(u))=\hat{x}(u)$. Moreover, 
\[
\pi \left( \phi_{n\tau, u}(x(u))\right) =\hat{\phi}_{n\tau, u}(\hat{x}%
(u))\in \hat{x}(u)\hat{G}^{0}=\pi(x(u)G^{0})
\]
\[
\implies \phi_{n\tau, u}(x(u))\in \pi^{-1}\left( \pi(x(u)G^{0})\right)
=x(u)G^{0}N^{0}=x(u)G^{0}%
\]
and hence
\[
\phi_{n\tau, u}(x(u))\in x(u)G^{0}, \; \; \mbox{ for all }\; \;n\in \mathbb{Z}.
\]

For the boundedness of the whole orbit $\{ \phi_{t, u}(x(u)), \;t\in \mathbb{R}\}$, let us consider
$t\in \mathbb{R}$ and write it as $t=n\tau+r$ with $n\in \mathbb{Z}$ and
$|r|\in[0, \tau)$. Then,
\[
\phi_{t, u}(x(u))=\phi_{r, u}(\phi_{n\tau, u}(x(u))\in \phi_{r, u}\left(
x(u)G^{0}\right) =\phi_{r, u}(x(u))\cdot G^{0},
\]
showing that
\[
\left \{ \phi_{t, u}(x(u)), \; \;t\in \mathbb{R}\right \} \subset \mathcal{A}%
_{\tau}(x(u))G^{0}\cup \mathcal{A}^{*}_{\tau}(x(u))G^{0},
\]
and consequently that $\left \{ \phi_{t, u}(x(u)), \; \;t\in \mathbb{R}\right \} $
is bounded.
\end{proof}

\bigskip

The next lemma shows that the orbits associated with $\mathcal{U}%
_{\mathrm{per}}$ are contained in $\overline{\mathcal{C}}$.

\begin{lemma}
\label{inQ}  For any $u\in \mathcal{U}_{\mathrm{per}}$ it holds that $\phi_{t,
u}\left( x(u)G^{0}\right) \subset \overline{\mathcal{C}}$, for any $t\in \mathbb{R}$.
\end{lemma}

\begin{proof}
Let us first prove that $x(u)G^{0}\subset \overline{\mathcal{C}}$. By Proposition \ref{x(u)} we
have that,
\[
\phi_{n\tau, u}(x(u))\in x(u)G^{0}, \; \; \mbox{ for all }\; \;n\in \mathbb{Z},
\]
where $\tau>0$ is the period of $u$. Decomposing $x(u)=gh$ with $g\in G^{+}$ and $h\in G^{-}$ gives us that
\[
\varrho(\phi_{n\tau, u}(x(u)), \phi_{n\tau, u}(g))=\varrho(\phi_{n\tau,
u}(g)\varphi_{n\tau}(h), \phi_{n\tau, u}(g))=\varrho(\varphi_{n\tau}(h),
e)\rightarrow0 \; \; \; \mbox{ as }\; \; \;n\rightarrow+\infty.
\]
In particular, $\phi_{n\tau, u}(g)\rightarrow x(u)g_{1}$ for some $g_{1}\in
G^{0}$. However, the fact that $G^{+, 0}\subset \mathcal{A}(e)$ implies by invariance
that $\phi_{n\tau, u}(g)\in \mathcal{A}(e)$ and hence $x(u)g_{1}\in
\overline{\mathcal{A}(e)}$.  Let then $g_{2}\in G^{0}$ arbitrary. By Theorem \ref{properties} it holds that
\[
x(u)g_{2}=x(u)g_{1}(g_{1}^{-1}g_{2})\in \overline{\mathcal{A}(e)}g_{1}^{-1}g_{2}=\overline{\mathcal{A}(e) g_{1}^{-1}g_{2}}\subset \overline{\mathcal{A}(e)}\; \implies \; \;x(u)G^{0}\subset \overline{\mathcal{A}(e)}.
\]
Consider now the decomposition $x(u)=h'g'$ with $h'\in G^{-}$ and $g'\in G^{+}$. By the same arguments as previously, we have that  $\phi_{-n\tau,
u}(h')\rightarrow x(u)h_{1}$, for some $h_{1}\in G^{0}$. Using that
$G^{-, 0}\subset \mathcal{A}^{*}(e)$ gives us that $x(u)h_{1}\in \overline{\mathcal{A}^{*}(e)}$ and again by Theorem \ref{properties} we get, for any
$h_{2}\in G^{0}$, that
\[
x(u)h_{2}= x(u)h_{1}(h_{1}^{-1}h_{2})\in \overline{\mathcal{A}^{*}(e)}h_{1}^{-1}h_{2}=\overline{\mathcal{A}^{*}(e)h_{1}^{-1}h_{2}}\subset \overline{\mathcal{A}^{*}(e)} \; \implies \; x(u)G^{0}%
\subset \overline{\mathcal{A}^{*}(e)}.
\]
Therefore,
\[
x(u)G^{0}\subset \overline{\mathcal{A}(e)}\cap \overline{\mathcal{A}^*(e)}\subset \overline{\mathcal{C}}.
\]
Consider now $t>0$. By invariance in positive time we already have that
$\phi_{t, u}(x(u)G^{0})\subset \overline{\mathcal{A}(e)}$. On the other
hand, for any $n\in \mathbb{Z}$ we have that
\[
\phi_{t, u}(x(u))=\phi_{t, u}\left( \phi_{-n\tau, u}\left( \phi_{n\tau,
u}(x(u))\right) \right) =\phi_{t-n\tau, u}\left( \phi_{n\tau, u}(x(u))\right)
\in \phi_{t-n\tau, u}\left( x(u)G^{0}\right) ,
\]
where for the last equality we used the cocycle property and the fact that $u$
is $\tau$-periodic.

Since $x(u)G^{0}\subset \overline{\mathcal{C}}\subset \overline{\mathcal{A}^{*}(e)}$, if we
take $n\in \mathbb{Z}$ such that $t-n\tau \leq0$, the invariance in negative
time of $\overline{\mathcal{A}^{*}(e)}$ implies that
\[
\phi_{t, u}(x(u))\in \phi_{t-n\tau, u}(\overline{\mathcal{A}^{*}(e)})\subset \overline{\mathcal{A}^{*}(e)}
\]
and hence, for any $g\in G^{0}$, we get that
\[
\phi_{t, u}(x(u)g)=\phi_{t, u}(x(u))\varphi_{t}(g)\in \overline{\mathcal{A}^{*}(e)}\varphi_{t}(g)=\overline{\mathcal{A}(e)\varphi
_{t}(g)}\subset \overline{\mathcal{A}^{*}(e)}\implies \phi_{t, u}%
(x(u)G^{0})\subset \overline{\mathcal{A}^{*}(e)},
\]
and consequently
\[
\phi_{t, u}(x(u)G^{0})\subset \overline{\mathcal{C}}, \; \; \mbox{ for all }\; \;t\geq0.
\]
By arguing analogously, we get that $\phi_{t, u}(x(u)G^{0})\subset \overline{\mathcal{C}}$ for all
$t<0$ and so
\[
\phi_{t, u}(x(u)G^{0})\subset \overline{\mathcal{C}}, \; \; \mbox{ for all }\; \;t\in \mathbb{R},
\]
concluding the proof. 
\end{proof}

\bigskip

The previous lemma allows us to construct a continuous function from the set
of control functions $\mathcal{U}$ to $G^{+, -}$ as follows: 

Let $u\in
\mathcal{U}$ arbitrary. Since $\mathcal{U}_{\mathrm{per}}$ is dense in
$\mathcal{U}$ there exists $u_{k}\in \mathcal{U}_{\mathrm{per}}$ with
$u_{k}\rightarrow u$. Denote by $x_{k}:=x(u_{k})$ the unique point in $G^{+,
-}$ given by Proposition \ref{x(u)}. By Lemma \ref{inQ} the sequence
$(x_{k})_{k\in \mathbb{N}}$ is contained in $\overline{\mathcal{C}}$ and hence is bounded. Moreover,
if $x\in G^{+, -}$ is such that $x_{k_{n}}\rightarrow x$, Lemma \ref{inQ} also
implies that
\[
\phi_{t, u}(x)=\lim_{n\rightarrow+\infty}\phi_{t, u_{k_{n}}}(x_{k_{n}})\in \overline{\mathcal{C}},
\; \; \mbox{ for all }\; \;t\in \mathbb{R}.
\]
Consequently, $\{ \phi_{t, u}(x), \; \;t\in \mathbb{R}\} \subset \overline{\mathcal{C}}$ is a bounded
orbit and by Proposition \ref{unique} the point $x\in G^{+, -}$ is the unique
point with such property. Therefore, the sequence $(x_{k})_{k\in \mathbb{N}}$
is bounded and has only one adherent point implying that $(x_{k}%
)_{k\in \mathbb{N}}$ converges to a point $x(u)\in G^{+, -}$.

We can now state and prove the main result of this paper.

\begin{theorem}
If a LCS $\Sigma_G$ admits a bounded control set $\CC$ with $e\in\inner\CC$ then all the bounded orbits of $\Sigma_{G}$ are contained in $\overline{\mathcal{C}}$.
\end{theorem}

\begin{proof}
Let $h\in G$ and decompose it as $h=xg$ with $x\in G^{+, -}$ and $g\in G^{0}$.
If for some $u\in \mathcal{U}$ we have that $\{ \phi_{t, u}(h), \;t\in
\mathbb{R}\}$ is bounded, then
\[
\phi_{t, u}(x)=\phi_{t, u}(hg^{-1})=\phi_{t, u}(h)\varphi_{t}(g^{-1})\in \{ \phi_{t,
u}(h), \; \;t\in \mathbb{R}\} G^{0},
\]
showing that $\{ \phi_{t, u}(h), \;t\in
\mathbb{R}\}$ is also bounded. Since $x\in G^{+, -}$ we have by Proposition \ref{unique} that $x=x(u)$ which by Lemma \ref{inQ} and the previous discussion
imply that
\[
\{ \phi_{t, u}(h), \;t\in \mathbb{R}\}=\{ \phi_{t, u}(x(u)g), \;t\in
\mathbb{R}\} \subset \overline{\mathcal{C}},
\]
as stated.
\end{proof}

\subsection{Consequences of the main result}

Let us define the \textit{lift} of the control set $\mathcal{C}$ as the set
\[
L(\mathcal{C}):=\{(u, g)\in \mathcal{U}\times G; \; \; \phi(\mathbb{R}, g,
u)\subset \overline{\mathcal{C}}\}.
\]

In this section we show that the lift of the bounded control set of a LCS is the continuous image of $\UC\times G^0$. 

By the main result, the map $(u, g)\in\UC\times G\mapsto x(u)g$ is continuous and the orbit $\{\phi_{t, u}(x(u)g), \;\;t\in\R\}$ is contained in $\overline{\mathcal{C}}$. Therefore,  
$$H:\mathcal{U}\times G^{0}\rightarrow L(\mathcal{C}), \; \; \; \; \; \;(u,
g)\mapsto(u, x(u)g),$$
is a well defined continuos map. Moreover, 
\[
(u, x)\in L(\mathcal{C})\; \; \; \iff \; \; \; \phi(\mathbb{R}, x, u)\subset
\overline{\mathcal{C}}\; \; \; \iff \; \; \; \{ \phi_{t, u}(x), \;t\in \mathbb{R}%
\} \; \; \mbox{ is bounded }\; \; \; \iff \; \; \;(u, x)\in H(\,\mathcal{U}\times G^{0}).
\]
Hence $H(\,\mathcal{U}\times G^{0})= L(\mathcal{C})$ and, since $\mathcal{U}\times G^{0}$ is compact and $H$ is continuous, $H$ is a
closed map and hence a homeomorphism between $\mathcal{U}\times G^{0}$ and $L(\mathcal{C})$. 

Let us assume that $G^{+, -}$ is a subgroup of $G$. Since $G$ is decomposable,
it holds that $G^{+, -}$ is a normal subgroup of $G$ and hence the
canonical projection $\pi:G\rightarrow G/G^{+, -}$ induces a LCS on
$G/G^{+, -}$. Moreover, the fact that $G^0\cap G^{+, -}=\{e\}$ implies that
$\pi|_{G^{0}}$ is a group isomorphism and hence we can assume w.l.o.g. that
$G/G^{+, -}=G^{0}$. Let us denote by $\Sigma_{G^{0}}$ the LCS
induced by $\pi$ on $G^{0}$ and by $\Phi^{0}$ its
control flow. The next result shows that under the extra assumption that
$G^{+, -}$ is a subgroup the dynamics of $\Phi_{t}|_{L(\mathcal{C})}$ is
basically the same as the one from the $\Phi^{0}_{t}$.

\begin{theorem}
\label{main}  If $G^{+, -}$ is a subgroup, then $H$ conjugates $\Phi
_{t}|_{L(\mathcal{C})}$ and $\Phi^{0}_{t}$.
\end{theorem}

\begin{proof}
	By the cocycle property, 
	$$\phi_{t+s, u}(x(u))=\phi_{t, \theta_su}\left(\phi_{s, u}(x(u))\right).$$
	Hence, 
	$$\left\{\phi_{t, \theta_su}\left(\phi_{s, u}(x(u))\right), \;\;t\in\R\right\}\;\;\mbox{ is a bounded orbit}.$$
	In particular, we have that $\phi_{s, u}(x(u))=x(\theta_su)h$ for an unique $h\in G^0$. Furthermore, under the previous identification, it holds that
	$$h=\pi(h)=\pi(h(h^{-1}x(\theta_su)h))=\pi(x(\theta_su)h)=\pi(\phi_{s, u}(x(u)))=\pi(\phi_{s, u}(e))=\phi_{s, u}^0(e),$$
	showing that 
	$$\forall s\in\R, \;\;\;\phi_{s, u}(x(u))=x(\theta_su)\phi_{s, u}^0(e).$$
	Therefore, 
\[
H\left( \Phi_{s}^{0}(u, g)\right) =H\left( \theta_{s}u, \phi_{s, u}%
^{0}(g)\right) =\left( \theta_{s}u, x(\theta_{s}u)\phi_{t, u}^{0}%
(e)\varphi_{s}(g)\right)
\]
\[
=\left( \theta_{s}u, \phi_{s, u}(x(u))\varphi_{s}(g)\right) =\left( \theta
_{s}u, \phi_{s, u}(x(u)g)\right) =\Phi_{s}(H(u, g)),
\]
showing that $H$ conjugates $\Phi_{t}|_{L(\mathcal{C})}$ and $\Phi^{0}_{t}$ as stated.
\end{proof}

\begin{remark}
 In \cite[Theorem 3.4]{Ka1} the author shows that, under a hyperbolicity assumption, $L(\CC)$ is the graph of a continuous function whose domain is the set of control functions $\UC$. In the context of LCS's, such an assumption is equivalent to $G^0=\{e\}$. Therefore our previous result shows that when the system is not necessarily hyperbolic but has compact central manifold (that is $G^0$ is a compact subgroup) we still have a characterization of $L(\mathcal{C})$ in terms of the central manifold and the set $\UC$.
\end{remark}

\subsubsection*{Acknowledgements}

The first author was supported by Proyecto Fondecyt n$^{o}$ 1190142, Conicyt, Chile and the second author was supported by Fapesp grants n$^{o}$ 2020/12971-4 and 2018/13481-0, and partially by CAPES grant no. 309820/2019-7.

\end{document}